\documentclass{amsart}
\usepackage{amsmath}
\usepackage{amssymb}
\usepackage{graphicx}
\usepackage{amsthm}
\usepackage{verbatim}
\usepackage{color}

\usepackage{amsfonts}
\usepackage[all]{xy}

\theoremstyle{plain}

\newtheorem{theorem}{Theorem}[subsection]
\newtheorem{lemma}[theorem]{Lemma}
\newtheorem{proposition}[theorem]{Proposition}

\theoremstyle{remark}
\newtheorem{remark}[theorem]{Remark}

\theoremstyle{definition}
\newtheorem{definition}[theorem]{Definition}

\newtheorem*{notation}{Notation}

\makeatletter

\newcommand{\Rmnum}[1]{\expandafter\@slowromancap\romannumeral #1@}
\makeatother

\begin{document}


\title{The equivalence between Feynman transform and Verdier duality}

\author{Hao Yu}

\address{Department of Mathematics, University of Minnesota, Minneapolis, MN, 55455, USA}
\email{dustless2014@163.com}




\begin{abstract}

The equivalence between dg duality and Verdier duality has been
established for cyclic operads earlier. We propose a generalization of
this correspondence from cyclic operads and dg duality to twisted
modular operads and Feynman transform. Specifically, for each twisted
modular operad $\mathcal{P}$ (taking values in dg-vector spaces over a field $k$ of characteristic 0), there is a certain sheaf $\mathcal{F}$
associated with it on the moduli space of stable metric
graphs such
that the Verdier dual sheaf $D\mathcal{F}$ is associated with the Feynman transform $F\mathcal{P}$ of $\mathcal{P}$. In the course of the proof, we also prove a relation between cyclic operads and modular operads originally proposed in the pioneering work of Getzler and Kapranov; however, to the best knowledge of the author, no proof has been
given in any literature. This
geometric interpretation in operad theory is of fundamental
importance. We believe this result will illuminate many aspects of the theory of modular
operads and find many applications in the future. We illustrate an application of this result,  giving another proof on the homotopy properties of Feynman transform, which is quite intuitive and simpler than the original proof.

\end{abstract}

\maketitle
\section{Introduction}

The equivalence between cyclic operads with dg duality and Verdier
duality for certain sheaves on the moduli space of stable graphs has been proposed and proved by  Lazarev and Voronov \cite{lv}.
A {first
result in this direction} appeared earlier in the pioneering work of
{Ginzburg} and Kapranov (cf. section 3 in \cite{gk}), in which the authors
showed that dg-duality for operads has a geometric interpretation
in terms of Verdier duality for sheaves on buildings, i.e., the space of
metric trees. 
Specifically, Lazarev and Voronov
prove that for any cyclic operad
$\mathcal{P}$ there is a certain sheaf $\mathcal{F}$ associated with
it on the moduli space of stable graphs such that {the
  Verdier dual $D\mathcal{F}$ of $\mathcal{F}$ is associated with the
  dg dual $D\mathcal{P}$ of this operad}.  This result was motivated
by the relations of graph homology introduced by Kontsevich
(\cite{ko}) for different kinds of Koszul dual operads, and it gives
an important conceptual explanation of the appearance of graph
cohomology of both the commutative and Lie types in computations of
the cohomology of the outer automorphism group of a free group (see
\cite{cv}, \cite{p} and \cite{lv} for more details).
One hopes that the similar result holds for twisted modular operads and Feymann transforms, which are higher genus generalizations of cyclic operads and dg dualities. In this paper, we show that this is indeed true. More precisely, we show that a general correspondence holds for twisted modular operads with Feynman transform and Verdier duality on the moduli space of stable metric
graphs. In particular, the result in \cite{lv} can be considered as a special case of ours. {The extension is desired, 
as the space the previous correspondence built on is most naturally suited for modular operads context, albeit the methods for both proofs are similar}. Technically, our proof is almost a direct extension of the proof for cyclic operads, along with careful manipulation of orientation factors, complemented with a uniqueness result. Our method of proof differs from the one in \cite{lv} in three main aspects. First, in cyclic operads case the sheaves corresponding to graphs of genus greater than 0 are set to 0, while in modular operads case they are certainly not 0 and need to be calculated as a whole. Secondly, there is no "twist" in the result for cyclic operad (reflecting the fact that any twist on trees are trivial), while in our case we are dealing with twisted modular operads. The third, we further prove a uniqueness result for this correspondence, while in \cite{lv} it is not available. Although this extension is expected, to the best knowledge of the author, this is the first time this equivalence is clearly stated and rigorously proved.  

Dualities in other related structures, such as  Koszul duality for PROPs (\cite{v}), etc., have been proposed as well. 
This would be of fundamental importance as it would 
illuminate many aspects of the theory of operads (reformulating algebraic structures into geometric languages has already been shown to be powerful in many areas in mathematics).

As an example of the usefulness of our equivalence, we present a simple and intuitive proof of the property
that the Feynman transform is a homotopy functor and has a homotopy inverse. 
More examples and applications will be left as a subject to a future work. 

The paper is organized as follows. In section 2, for self-completeness, we discuss preliminaries
relevant to twisted modular operads, Feynman transform, combinatorial
formulation of Verdier duality, etc.
In section 3, we prove the main theorem, the equivalence between
Feynman transform and Verdier duality. The cyclic version of this equivalence is shown {to follow}
from our result. A simple proof on the homotopy property of
Feynman transform is given. Finally, we discuss future projects in the direction of this paper. 

\begin{notation}
Throughout this paper, $k$ will denote a {field of
  characteristic 0. A dg-vector space or chain complex over $k$ is a sequence $\dots
  \rightarrow V_{i+1} \rightarrow V_i \rightarrow V_{i-1} \rightarrow
  \dots$, where each $V_i$ is a vector space over $k$, and with a
  linear operator $d: V_i \rightarrow V_{i-1}$ such that $d^2 =
  0$ which is called the \emph{differential}. Similarly, a cochain complex is a sequence $\dots
  \rightarrow V^{i-1} \rightarrow V^i \rightarrow V^{i+1} \rightarrow
  \dots$, where each $V^i$ is a vector space over $k$, and with a differential $d: V^i \rightarrow V^{i+1}$ such that $d^2 =
  0$ }

The contravariant linear dual functor $V \rightarrow V^*$ taking the category of
chain complexes to the category of chain complexes is defined as
follows:
\begin{eqnarray}
V^*_i = (V_{-i})^*
\end{eqnarray}
with the differential $\quad d^*: V^*_i \rightarrow V^*_{i-1}$ such that the equality $d^*(v^*)(w) = v^*(dw)$ holds for $v^* \in (V_{-i})^*$ and $w \in V_{-i+1}$.  
A different contravariant dual functor $V \rightarrow
V^{\vee}$ taking the category of chain complexes to the category of
cochain complexes is defined as follows:
\begin{eqnarray}
(V^{\vee})^i = V^*_{i}
\end{eqnarray}
with the differential $d^{\vee}: (V^{\vee})^i \rightarrow (V^{\vee})^{i+1}$ such that the equality $d^{\vee}(v^*)(w)=v^*(dw) $ holds for $v^* \in V^*_i$ and $w \in V_{i+1}$.
 Clearly, $V^{* \vee} \cong V^{\vee *}$ and
$(V^{\vee *})^i \cong V_{-i}$.

Given a chain complex $V$, its \textit{shift} $V[n]$ is a chain
complex with $V[n]_i = V_{i+n}$ 
And given a cochain complex $V$, its \textit{shift} $V[n]$ is a cochain complex $(V[n])^i = V^{i+n}$.
{ Note that $(V[1])^* \cong V^*[-1]$ and
  $(V[1])^\vee \cong V^\vee [1]$.} The \textit{suspension} $\Sigma
V$ of $V$ is a complex with $\Sigma V_i = V_{i-1}$ or $\Sigma V^i =
V^{i-1}$ in the category of chain complexes or cochain complexes,
respectively.

An ungraded vector space V could be assumed to lie in degree 0, and it
will be clear from the context whether this (trivial) grading is
considered homological or cohomological. If dim $V = n$ we will call
{the one-dimensional graded vector space $Det(V ) =
  S^{n}(V[-1])=\Lambda^{n}(V)[-n]$, concentrated in dimension $n$, the
  determinant of $V$.}  We must emphasize that the grading of $Det(V)$ is different from the grading convention in \cite{gk2} and \cite{b}, whereas they use the grading of $-n$ instead. 
Note that $Det(V)^*[-2n]\cong Det(V^*)$. If $I$ is a finite set
{of} cardinality $|I|$, we define $Det(I) := Det(k^{|I|})$. We
use the negative power of {a one-dimensional} graded vector
space to denote the positive power of corresponding *-dual, so that
$Det^{-p}(V) = ((DetV)^*)^{\otimes p}$.

For a finite collection $\{V_\alpha | \alpha \in I\}$ of
finite-dimensional vector spaces, we have a natural identification
\begin{equation}
\bigotimes_{\alpha \in I}V_\alpha[-1] \cong Det(I) \otimes \bigotimes_{\alpha \in I}V_\alpha
\end{equation}
For a simplex $\tau$, the symbol $Det(\tau)$ denotes the determinant of the set of vertices of $\tau$.
When the ground filed $k=R$, a choice of a nonzero element in $Det(\sigma)$ up to a positive real factor is equivalent to providing $\sigma$ with an orientation in the usual sense.

Finally, throughout the rest of this paper, all the operads take values in symmetric monoidal category of dg-vector spaces over a fixed ground field $k$ of characteristic 0, unless otherwise specified.

\end{notation}

\section{Preliminaries}

We present preliminary results relevant to modular operads, Feynman
transform, Verdier duality and {other notions}. The experts are certainly free to skip over the first few subsections. We will omit proofs for majority of propositions discussed in this section, referring the interested readers to original papers, for example,    \cite{b},\cite{lv},\cite{gk2},\cite{gk3},\cite{snpr} and references therein for more details.
\subsection{(Twisted) modular operads and Feynman transform}
\subsubsection{Stable graphs}
{The notion of \textit{stable graphs} is intimately linked to the notion of \textit{stable curves}. Stable curves are used in the theory of compactification of the moduli space of smooth algebraic curves (aka Deligne-Mumford moduli space) (\cite{dm})}. 
Non-technically speaking, a \textit{stable graph} is a graph encoding incidence relations between various strata of this compactified moduli space, whence the name. The precise definition is as follows.

\begin{definition}(Stable graph) A \textit{graph} $G$ is specified by a finite set $H(G)$ (whose elements are called half edges), an involution $\sigma: H(G) \rightarrow H(G)$ such that $\sigma^2 = I$ and a partition $\lambda$ of $H(G)$ (By a partition of a set, we mean a
disjoint decomposition into several unordered, possibly empty, subsets, called blocks). We say that two half edges $a,b \in H$ meet if and only if they are equivalent under the partition $\lambda$. The $\textit{vertices}$ of $G$ are the blocks of the partition $\lambda$, and the set of them is denoted by $V(G)$. The subset of $H(G)$ corresponding to a vertex $v$ is denoted by $Leg(v)$. Its cardinality is called the valence of $v$, and denoted $n(v)$. The \textit{edges} of $G$ are the pairs of half edges forming the two-cycle of $\sigma$, and the set of them is denoted by $E(G)$. The cardinality of $E(G)$ is denoted by $|G|$. The legs of $G$ are the fixed-points of $\sigma$, and the set of them is denoted by $Leg(G)$ (note its difference with $Leg(v)$). cf. for example, \cite{gk3} for precise definitions. 

A \textit{stable graph} is a graph with labels on $V(G)$, $g: V(G) \rightarrow Z_{\ge 0}$, called the \textit{genus} of vertices, satisfying $2g(v)-2+n(v)>0$ for each vertex $v$. The geometric realization of a graph is a CW complex {denoted by $\mathcal{T}(G)$}. The genus of the graph $G$ is defined as $g(G) = \sum_{v\in V(G)}g(v) + dim_k(H_1(\mathcal{T}(G),k))$, where $H_1(\mathcal{T}(G),k)$ is the first homology group of the geometric realization of $G$ with coefficient $k$. {A \textit{tree} is a graph of genus 0}. {We denote by $n(G)$ the cardinality of $Leg(G)$} (cf. \cite{gk}).
\end{definition}


\begin{definition}(Morphism of graphs)
Let $G_0$ and $G_1$ be two graphs. A \emph{morphism} $f: G_0 \rightarrow G_1$ is an injection $f^*:H(G_1) \rightarrow H(G_0)$ such that
\begin{enumerate}
\item $\sigma_0 \circ f^* = f^* \circ \sigma_1$, where $\sigma_i, i = 0, 1$ are the involutions of $H(G_i)$.
\item $\sigma_0$ acts freely on the complement of the image of $f^*$ in $H(G_0)$ (i.e., $G_1$ is obtained from $G_0$ by contracting a subset of its edges).
\item two half edges $a$ and $b$ in $G_1$ meet if and only if there is a chain $(x_0,\ldots x_k)$ of half edges in $G_0$ such that $f^*a = x_0,\sigma_0x_{i-1}$ and $x_i$ meet for all $1\le i \le k$, and $f^*b = \sigma_0x_{k}$.
\end{enumerate}
\end{definition}
\begin{remark}
Generally speaking, the conditions (2) and (3) indicate that the morphisms of graphs are the compositions of isomorphisms and contractions of a set of edges, which satisfy the compatibility condition (1) (see also (2.14) of \cite{gk3}).
\end{remark}
Given such a morphism, we denote $f^{-1}(v)$ for a vertex $v \in G_1$ to be a stable subgraph of $G_0$, called the preimage of $v$,  consisting of those half edges in $G_0$ which are connected to a half edge in $Leg(v)$ by a chain of edges in $G_0$ being contracted by the morphism.

\begin{definition}(Contraction of graphs)
Let $I$ be a set of edges of a stable graph $G$. We define the contraction of $G$ along $I$ to be a unique stable graph, denoted by $G/I$,
satisfying the following properties.
\begin{enumerate}
\item $H(G/I)$ is obtained from $H(G)$ by deleting helf edges consisting of edges in $I$.
\item the inclusion $H(G/I) \hookrightarrow H(G)$ is a morphism of graphs $f_{G, I}: G \rightarrow G/I$.
\end{enumerate}
\end{definition}



\subsubsection{Operads and Cyclic operads}
\begin{definition}($S$-module) A \textit{$S$-module} is a collection of chain complexes of vector spaces over $k$ $\mathcal{P}(n)$ for $n \ge 0$, together with an $S_n$ action, where $S_n$ is a permutation group on $\{1,2,\ldots,n\}$.
A \textit{morphism of S-modules} $f : E \rightarrow F$ is a sequence of $S$-equivariant morphisms $f(l) : E(n) \rightarrow  F(n)$. We denote the category of S-module by $S-Mod$.

One can extend the definition of $\mathcal{P}$ to groupoid of finite sets and bijections via left Kan extension:

\begin{equation}
\mathcal{P}(I) = (\bigoplus_{\textrm{bijection}: I \leftrightarrow \{1,2,\ldots,n\}} \mathcal{P}(n))_{S_n}
\end{equation}
\end{definition}


\begin{definition}(Operad) \label{operad} A \textit{unital operad} (an \textit{operad} for short) is a $S$-module $P$ together with a family of structure morphisms:
\begin{equation}\label{operad_morphisms}
\gamma_{l;m_1,...,m_l}: P(l)\otimes P(m_1)\otimes \cdots \otimes P(m_l) \rightarrow P(m_1+\cdots +m_l)
\end{equation}
and unit:
\begin{equation}
\eta: 1 \rightarrow P(1)
\end{equation}
satisfying the axioms of equivariance, associativity and unit.
A \textit{morphism of operads} is a morphism of $S$-module compatible with the structure morphisms.
\end{definition}

For all $l \in N$, recall that the group $S^+_l := Aut\{0, 1, . . . , l\}$ contains $S_l$ as a subgroup, and it is generated by $S_l$ and the cyclic permutation of order $l+1, \tau_l : (0,1,\ldots,l) \rightarrow (1,2,\ldots,l,0)$.
\begin{definition} (Cyclic $S$-module) A \textit{cyclic $S$-module} is a collection of chain complexes of vector spaces over $k$ $\mathcal{P}(n)$ for $n \ge 0$, together with an $S^+_n$ action. We denote the category of cyclic $S$ module by $S^+-Mod$. Forgetting the cyclic permutation we have a functor
\begin{equation}
\rho: S^+-Mod \rightarrow S-Mod
\end{equation}
\end{definition}
\begin{definition} (Cyclic operad) A \text{cyclic operad} is a cyclic $S$-module $\mathcal{P}$ whose underlying $S$-module $\rho \mathcal{P}$ module has a structure of operad compatible with the action of $S^+_n$.
\end{definition}
\subsubsection{Modular operads}
A cyclic $S$-module $\mathcal{P}$ such that $\mathcal{P}(n)=0$ for $n\le 1$ can be regarded as a stable $S$-module $\tilde{\mathcal{P}}$ with $\tilde{\mathcal{P}}(0,n)=\mathcal{P}(n-1), \tilde{\mathcal{P}}(g,n)=0$ for $g\ge 1$. This point will be used in the later part of the paper.
\begin{definition}(Stable $S$-module) A \textit{stable $S$-module} is a collection of chain complexes of vector spaces over $k$ $\mathcal{P}(g,n)$ for $g,n\ge 0$, such that $\mathcal{P}(g,n)=0$ for $2g+n-2\le 0$, together with an $S_n$ action, where $S_n$ is a permutation group on $\{1,2,\ldots,n\}$.
A \textit{morphism of S-modules} $f : E \rightarrow F$ is a sequence of $S$-equivariant morphisms $f(l) : E(g,n) \rightarrow  F(g,n)$. We denote the category of stable S-module by $sS-Mod$.

Again, we extend the definition of $\mathcal{P}$ to groupoid of finite sets and bijections via left Kan extension:
\begin{equation}
\mathcal{P}(g,I) = (\bigoplus_{\textrm{bijection}: I \leftrightarrow \{1,2,\ldots,n\}} \mathcal{P}(g,n))_{S_n}
\end{equation}
\end{definition}
 For a stable graph $G$, we define
\begin{equation}
\mathcal{P}((G)) = \bigotimes_{v\in V(G)} \mathcal{P}(g(v),Leg(v))
\end{equation}
\begin{remark}
Note that it seems we have two different notations $P((G))$ and $P(g, n)$ for modular operads. This can be understood as follows: if $G$ is a one-vertex graph $*_{g, n}$ of genus $g$, valence $n$ and with no edges, then $P(g, n)$ is by definition the same as $P((*_{g, n}))$, so by this understanding we simply use the more natural symbol $P(g, n)$ to refer to $P((*_{g, n}))$. If $P$ is a cyclic operad, we ignore $g$ and use the similar notations $P(n)$ and $P((T))$ for a stable tree $T$ to mean the same thing. 
\end{remark}
\begin{remark}
The notion of stable $S$-module differs from (cyclic) $S$-module in that it has one more parameter $g$. If the $S$-module is the underlying $S$-module of the modular operad of stable curves, $g$ correspondes to the genus of the curves. 
A cyclic $S$-module $\mathcal{P}$ such that $\mathcal{P}(n)=0$ for $n\le 1$ can be regarded as a stable $S$-module $\tilde{\mathcal{P}}$ with $\tilde{\mathcal{P}}(0,n)=\mathcal{P}(n-1), \tilde{\mathcal{P}}(g,n)=0$ for $g\ge 1$. This point will be used in the later part of the paper.
\end{remark}
\begin{definition}(Modular operad) \label{modular_operad} A \textit{modular operad} is a stable $S$-module $\mathcal{P}$ together with, for each pair $(G,b)$ where $G$ is a stable graph with $g(G)=g, n(G)=n$ and $b$ is a bijection between $Leg(G)$ and $\{1,2,\ldots,n\}$, composition maps:
\begin{equation}\label{m_morphisms}
\mu^\mathcal{P}_{G,b}: \mathcal{P}((G)) \rightarrow \mathcal{P}(g,n)
\end{equation}
satisfying associativity with respect to the compositions in the category of stable graphs and equivariance with respect to the $S_n$ actions of relabelling the legs.
In particular, a \emph{cyclic operad} $P$ such that $P(n) = 0$ for $n\le 1$ can be regarded as a modular operad $\mathcal{\tilde{P}}$ with $\mathcal{\tilde{P}}(0,n)=\mathcal{P}(n-1), \mathcal{\tilde{P}}(g,n) = 0$ for $g$ $\ge$ $1$.
\end{definition}
\begin{remark} This definition of modular operad is from Barannikov(\cite{b}), which is equivalent to the definition in the original work of {Ginzburg} and Kapranov (\cite{gk}). We adopt this definition because it is more natural for the purpose of the paper and will simplify the proofs.
\end{remark}
Let $I$ be a finite set, similarly we can extend the definition of composition maps to stable graphs $G$ with legs marked by $I$:
\begin{equation}
\mu^\mathcal{P}_{G}: \mathcal{P}((G)) \rightarrow \mathcal{P}(g,I)
\end{equation}
Due to $S_n$ equivariance, the composition maps do not depend on the choice of bijection between $I$ and $\{1,2,\ldots,n\}$.

\subsubsection{Twisted modular operads}

The notion of twisted modular operads was introduced in \cite{gk}  to allow more flexibility to compose different operations.



\begin{definition}(Cocycle) A \textit{cocycle} is a functor $\rho$ from the category $Iso(g,n)$, $g,n \ge 0$, of the isomorphism classes of stable graphs of genus $g$ and valence $n$ and their isomorphisms to the Picard tensor symmetric category of invertible graded one dimensional vector spaces $\mathcal{P}\mathcal{I}\mathcal{C}$, such that:
\begin{enumerate}
\item To each morphism $f: G_1 \rightarrow G_2$ is associated an isomorphism:
\begin{equation}
\label{cocycle}
\phi_{G_1\rightarrow G_2}:\rho(G_2)\otimes \bigotimes_{v\in V(G_2)} \rho(f^{-1}(v)) \rightarrow \rho(G_1)
\end{equation}
\item For graphs $*_{g,n}$ with one vertex and no edges of genus $g$ and valence $n$, $\rho(*_{g,n}) = 1$, the unit object in $\mathcal{P}$$\mathcal{I}$$\mathcal{C}$.
\end{enumerate}
\end{definition}

For a cocycle $\rho$,  $\rho^{-1}$ is also a cocycle, and if $\rho_1,\rho_2$ are cocycles, $\rho_1 \otimes \rho_2$ is a cocycle. A typical example of a cocycle is $\rho(G) = Det(E(G))[2m]$, where $m=|E(G)|$. It is called  \textit{dualizing cocycle} and is denoted by $R$. Sometimes we also use $Det(G)$ to denote it in this paper. If $\rho$ is a cocycle, the \textit{dual} of $\rho$ is defined as the cocycle $R\otimes \rho^{-1}$, and is denoted by $\rho^{\vee}$. Another example of a cocycle is
\begin{equation}
\rho(G) = Det(H_1(\mathcal{T}(G),k))[2h]Det^{-1}(Leg(G))[-2l]
\end{equation}
where $h=dim_kH_1(\mathcal{T}(G),k)$ and $l=n(G)$.

\begin{definition}(Twisted modular operad) A $\rho$ twisted modular operad, or a modular $\rho$-operad, is a stable $S$-module $
\mathcal{P}$ together with, for each stable graph $G$, composition maps:
\begin{equation}
\mu^\mathcal{P}_{G}: \rho(G) \otimes \mathcal{P}((G)) \rightarrow \mathcal{P}(g,n)
\end{equation}
satisfying the associativity and equivariance conditions  parallel to the definition of modular operads.
\end{definition}

Given a morphism $f: G_1 \rightarrow G_2$, we define the \textit{structure maps}  $\mu^\mathcal{P}_{G_1 \rightarrow G_2}: \rho(G_1) \otimes \mathcal{P}((G_1)) \rightarrow \rho(G_2)\otimes \mathcal{P}((G_2))$ as follows:
\begin{equation}
\mu^\mathcal{P}_{G_1\rightarrow G_2}:= \bar{\phi}_{G_1\rightarrow G_2} \bigotimes \otimes_{v\in V(G_2)} \mu^\mathcal{P}_{f^{-1}(v)}
\end{equation}
where $\bar{\phi}_{G_1\rightarrow G_1}: \rho(G_1)\otimes_{v\in V(G_2)}\rho(f^{-1}(v))^{-1} \rightarrow \rho(G_2)$ is an isomorphism induced by the isomorphism $\phi_{G_1 \rightarrow G_2}$.

\subsubsection{Feynman transform}

\begin{definition}(Freely generated modular $\rho$-operad) Given a stable $S$-module $\mathcal{P}$, the modular $\rho$-operad $M_{\rho}\mathcal{P}$ freely generated by $\mathcal{P}$ is defined as follows:
\begin{equation}\label{free_twist}
      M_{\rho}\mathcal{P}(g,n) = \bigoplus_{G \in \Gamma(g,n)} (\rho(G) \otimes \mathcal{P}((G)))_{Aut(G)}
\end{equation}
where $\Gamma(g,n)$ is the set of isomorphism classes of pairs $(G,b)$ where $G$ is a stable graph with $g(G)=g, n(G)=n$ and $b$ is a bijection from $Leg(G)$ to $\{1,2,\ldots,n\}$.

The composition maps $\mu^\mathcal{M_{\rho}\mathcal{P}}_{G}$ of $M_{\rho}\mathcal{P}(g,n)$ are the canonical embeddings.

\end{definition}

One can extend the definition to modular operads with inputs labeled by any finite set $I$, by using covariant trick.
\begin{equation}
\label{MP}
M_{\rho}\mathcal{P}(g, I) = \bigoplus_{G \in \Gamma(g,n,I)} (\rho(G) \otimes \mathcal{P}((G)))_{Aut(G)}
\end{equation}
where $\Gamma(g,n,I)$ is the set of isomorphism classes of $G$ where $G$ is a stable graph with $g(G)=g, n(G)=n$ and legs marked by $I$.

\begin{definition}(Feynman transform) A Feynman transform of a modular $\rho$-operad $\mathcal{P}$ is a modular $\rho^{\vee}$-operad $F_{\rho}\mathcal{P}$, defined in the following way. As a stable $S$-module, forgetting the differential, $F_{\rho}\mathcal{P}$ equals $M_{\rho^{\vee}}\mathcal{P}$, the underlying stable $S$-module of the modular $\rho^\vee$-operad freely generated by stable $S$-module \{$\mathcal{P}(g,n)^{*}$\} according to \eqref{free_twist}. The differential on $F_{\rho}\mathcal{P}$ is the sum $d_{F} = \partial_\mathcal{P}^{*} + \partial_\mu$ of differential $\partial_\mathcal{P}^{*}$ dual to the differential $\partial_\mathcal{P}$ of $\mathcal{P}$ and of differential $\partial_\mu$, whose value on the term $(\rho^{\vee}(G)\otimes \mathcal{P}((G))^{*})_{Aut(G)}$ is the sum over all equivalence classes of stable graphs $\tilde{G}$ such that $\tilde{G}/{e}\cong G$ of the map dual to the structure map $\mu^\mathcal{P}_{\tilde{G}\rightarrow G}$ multiplied by an orientation map:
\begin{equation}
\label{orientation}
\epsilon : R(\tilde{G}) \rightarrow R(G)
\end{equation}
which is given by multiplying by the natural basis element $e$ of $Det({e})$. Specifically,
\begin{equation}
\label{diffFeynman}
\partial_\mu|_{(\rho^{\vee}(G)\otimes \mathcal{P}((G))^*)_{Aut(G)}} = \epsilon \bigotimes \sum_{\tilde{G}/{e}\cong G}  (\mu^\mathcal{P}_{\tilde{G}\rightarrow G})^{*} 
\end{equation}

The following diagram illustrates a direct summand of $\partial_\mu$ for a specific $\tilde{G}$:

\begin{equation*}
\xymatrix{
 R(\tilde{G}) \otimes \rho (\tilde{G})^* \otimes \mathcal{P}((\tilde{G}))^* \ar[rrr]^{\epsilon \otimes (\mu_{\tilde{G}\rightarrow G})^*} \ar[d]
 & & & R(G) \otimes \rho (G)^* \otimes \mathcal{P}((G))^* \ar[d] \\
\rho^{\vee}(\tilde{G})\otimes \mathcal{P}((\tilde{G}))^* \ar@{.>}[rrr]^{\partial_\mu |_{\tilde{G}}} & & & \rho^{\vee}(G)\otimes \mathcal{P}((G))^*
}
\end{equation*}

\end{definition}
See \cite{gk3} for more details.

\subsubsection{Dg duality}
Assume now $\mathcal{O}$ is a cyclic operad. A duality theory for cyclic operads is discussed in \cite{lv}, and this duality has nice relations with Feynman transform for twisted modular operads. This is a crucial ingredient towards proving that the result in \cite{lv} can be derived from our result in the next section. 

\begin{definition}\label{cobar_duality}(cobar duality) A cobar dual operad of a (dg) cyclic operad $\mathcal{O}$ is a (dg) cyclic operad $\mathcal{B}\mathcal{O}$ defined as follows. As a cyclic $S$-module, forgetting the differential the $\mathcal{B}\mathcal{O}$ is freely generated by $(\mathcal{O}((T)))^{*}$ where $T$ is a stable trees, i.e., a stable graph of genus 0.
The differential of $\mathcal{B}\mathcal{O}$ is exactly the same as \eqref{diffFeynman} except that $G$ and $\tilde{G}$ are replaced by stable trees $T$ and $\tilde{T}$, respectively.
\end{definition}

\begin{definition}\label{dg_duality}(dg duality) A dg dual operad of a (dg) cyclic operad $\mathcal{O}$ is a (dg) cyclic operad $D\mathcal{O}$ defined as follows. As a cyclic $S$-module, forgetting the differential the $D\mathcal{O}$ is freely generated by $(s\mathcal{O}[-1]((T)))^{*}$ where $T$ is a stable tree, i.e., a stable graph of genus 0, and $s\mathcal{O}$ is \textit{cyclic-operad suspension}(\cite{gk2}) defined by:
\begin{equation}
s\mathcal{O}(n):=Det^{-1}(k^{n+1})[-2]\otimes \mathcal{O}(n)
\end{equation}
The differential of $D\mathcal{O}$ is exactly the same as \eqref{diffFeynman} except that the $G$ and $\tilde{G}$ are replaced by stable tree $T$ and $\tilde{T}$, respectively.
\end{definition}

\begin{remark} In general, the suspension $s$ for a modular operad $\mathcal{P}$ is defined as $s\mathcal{P}(g,n):=\Sigma^{-2(g-1)}\otimes Det^{-1}(k^{n})\otimes \mathcal{P}(g,n)$.
\end{remark}

\subsection{Combinatorial formulation of Verdier duality}

The other side of this picture of equivalence is Verdier duality. The most important observation is that, on a certain space, Verdier duality has a combinatorial reformulation. This formulation was presented in a slightly different context of space stratified into cells in \cite{kp}, where it continues to hold in our setting of simplicial complex (\cite{lv}).

\subsubsection{Verdier duality}

On a finite dimensional locally compact space $X$, the dualizing complex $\Omega$ is defined by $\Omega := p^{!}(k)$ via $p: X \rightarrow \bullet$ (point). \textit{Verdier duality} is a contravariant functor from the bounded derived category of sheaves on a finite dimensional locally compact space $X$ to itself, $D: D^b(X) \rightarrow D^b(X)$,
defined by
\begin{equation}
 D(\mathcal{F}) = RHom(\mathcal{F}, \Omega)
\end{equation}
which have the following properties:
\begin{itemize}
\item \label{verdier_duality} There is a natural transformation $D^2(\mathcal{F}) \rightarrow id$.  This is an isomorphism
$D^2(\mathcal{F}) \cong \mathcal{F}$
in the bounded derived category of sheaves of constructible cohomology.  
\item $D(Rf_{*}(\mathcal{F})) \cong Rf_{!}D(\mathcal{F})$ where $f$ is a continuous map from $X$ to $Y$ and $f_{*}, f_{!}$ are push forward and direct image with compact support functor, respectively.
\end{itemize}
Recall that a constructible sheaf is a sheaf of abelian groups over some topological space $X$, such that X is the union of a finite number of locally closed subsets on each of which the sheaf is a locally constant sheaf of dg-vector spaces over $k$ of finite dimension in each degree. And by bounded derived category of sheaves of \textit{constructible cohomology} we mean that it is a bounded derived category of sheaves for which all cohomology are constructible sheaves.  A striking fact is that when $X$ is a complex algebraic variety, the bounded derived category of constructible sheaves is the same as the bounded derived category of sheaves of constructible cohomology, which is proved in \cite{n}. 



\subsubsection{Verdier duality on (orbi-)simplicial complex}

\begin{definition}(Constructible sheaf over finite simplicial complexes) Let $X$ be a finite simplicial complex. A sheaf of dg-vector spaces over a ground field $k$ on $X$ is called \textit{constructible} if its restriction to each face is a constant sheaf.
\end{definition}

Consider the category whose objects are faces of $X$ and morphisms are inclusions of faces. A \emph{coefficient system} on $X$ is a covariant functor from this category to the category of dg-vector spaces. 

The \textit{open star} $U_\sigma$ of a simplex $\sigma$ in $X$ is the union of interior of  those faces of $X$ which contain $\sigma$. Any sheaf determines a contravariant functor from the poset $\{U_\sigma\}$ into the category of dg-vector spaces. Conversely, any constructible sheaf is completely determined by the corresponding functor. 
In \cite{lv}, the proposition (1.3) shows that there is a one-to-one correspondence between constructible sheaves and coefficient systems on $X$, derived from the fact that the category of faces of $X$ is opposite to the category of open stars of $X$. See \cite{lv} for full details.




Our results are based on the setting of orbi-simplicial complex. 
Let $X$ be a finite-dimensional simplicial complex, and $G$ a group acting properly discontinuously and simplicially on $X$, which means that for each $x\in X$ the stabilizer $G_x$ is a finite group and there is a neighborhood $U_x$ of $x$ such that $gU_x \cap U_x = \emptyset$, and for each simplex $\sigma \in X$ and $g \in G$ the image $g(\sigma)$ is another simplex of $X$ and $g: \sigma \rightarrow g(\sigma)$ is an affine map. Let $Y$ be the orbit space $Y = X/G $.  $Y$ is glued from \textit{orbi-simplices}, i.e., the quotients of simplices by the action of finite groups, and there is one $n$-dimensional orbi-simplex in $Y$ for each orbit of the action of $G$ on the set of $n$ simplexes of $X$.
\begin{definition}
A sheaf $\mathcal{F}$ on $Y$ is called \textit{constructible} if $f^{-1}\mathcal{F}$ is a constructible sheaf on $X$.
\end{definition}

{Before presenting the formulation of Verdier duality on simplicial complex, we recall (see the Notation section) that the gradings of $Det(V)$ in \cite{lv} and \cite{gk,b} are different. In \cite{lv} it is concentrated in degree $dim(V)$ while in \cite{gk} and \cite{b} it is concentrated in degree $-dim(V)$. As the readers have probably seen in the preceeding paragraphs (e.g. the orientation cocycle), we adopt the convention of the \textit{former} in this paper}.

\begin{theorem}\label{orbi-verdier_duality}
 Let $X$ be a finite-dimensional orbi-simplicial (resp. simplicial) complex. On $X$, let $\mathcal{F}$ be a constructible sheaf and $D\mathcal{F}$ its Verdier dual, and for an orbi-simplex (resp. simplex) $\tau$ of $X$, let $\{\mathcal{F}_{\sigma}\}, \{D\mathcal{F}_{\sigma}\}$ be the corresponding coefficient systems. Then $D\mathcal{F}$ is represented by the constructible complex $\sigma  \rightarrow \mathcal{F}_{\sigma}$, where $D\mathcal{F}_{\sigma}$ is the following cochain complex:
\begin{equation}
\bigoplus_{\tau \supset \sigma}(\mathcal{F}_{\tau}\otimes Det(\tau)[1])^*
\end{equation}
whose differential is dual to
\begin{equation}
\label{orbi_vd_restriction}
\mathcal{F}_\sigma\otimes Det(\sigma)[1] \rightarrow \sum_{\substack{\tau \supset \sigma \\ dim \tau = dim \sigma + 1}} \mathcal{F_{\tau}}\otimes Det(\tau)[1]
\end{equation}
where the last map is induced by inclusions $\sigma \hookrightarrow \tau$.

\end{theorem}
\begin{remark} This theorem is a combination of propositions 1.5, 1.6, 2.4 and 2.5  in \cite{lv}. 
	Note that the similarity between \eqref{orbi_vd_restriction} and \eqref{diffFeynman} is suggestive, which might imply a correspondence between Feynman transform and Verdier duality.
\end{remark}

\section{Correspondence between Feynman transform and Verdier duality}
\subsection{Theory of Correspondence}
\subsubsection{Main theorem} \label{moduli_space_graphs}
We first introduce the notion of moduli space of stable metric graphs from \cite{lv}.  This paragraph is nearly completely borrowed from \cite{lv} so the readers can refer to their paper for full details and clear exposition. 

A \textit{metric graph} is a graph together with a function $l: E(G) \rightarrow R_{+}$, where $R_{+}$ is the set of nonnegative reals; $l(e)$ is called the length of the edge $e$. We consider the stable metric graph with $\sum_{e \in E(G)} l(e) = 1$. The moduli space of such graphs is shown in \cite{lv} to be an orbi-simplicial complex, where each orbi-simplex corresponds to an isomorphism class of a stable graph,  and collapsing the edges corresponds to passing to the faces of corresponding orbi-simplex. The dimension of the orbi-simplex equals the cardinality of the edge set of the corresponding graph minus one: dim($\tau$) = $|E(\Gamma_{\tau})| - 1$, where $\tau$ is an orbi-simplex, $\Gamma_{\tau}$ is the stable graph corresponding to $\tau$. We denote this moduli space by $M$.
	Let $\bar{Y_g}$ be the subset and also a sub-simplicial complex of $M$ corresponding to the stable graphs of genus $g$. $M$ is the disjoint union of all $\bar{Y_g}$. Let $Y_g$ denote the subset (not a sub-simplicial complex) of $\bar{Y_g}$ corresponding to stable metric graphs whose genus of each vertex is 0 (the genus of the stable graph may be greater than 0 due to the existence of loops or the graph is cyclic). Let $i:Y_g \to \bar{Y_g}$ be the corresponding inclusion. For details on how the orbi-simplicial structures of $\bar{Y_g}$ and $M$ are constructed we refer the readers to the original paper (\cite{lv}).

For a modular $\rho$-operad $\mathcal{P}$, we define a constructible complex $\mathcal{F}^{\mathcal{P}}$ associated with it on $M$ as follows. Generally speaking, the complex of sheaf $\mathcal{F}^{\mathcal{P}}$ is the complex over $k$ generated by $\mathcal{P}^{\vee}$, where the form of compositions and the final configuration is labelled by the stable graph $\Gamma$. It is therefore a direct summand of the "dual" modular operad of $\mathcal{P}$ in degree $g = g(G), n = n(G)$. Specifically, for each orbi-simplex $\sigma$ corresponding to graph $\Gamma$, we set $\mathcal{F}^{\mathcal{P}}_{\sigma}$ to be the cochain complex dual to the chain complex $\rho(\Gamma)\otimes \mathcal{P}((\Gamma))$:
\begin{equation}
\label{sheaf_asso_operad}
\mathcal{F}^{\mathcal{P}}_{\sigma} = (\rho(\Gamma)\otimes \mathcal{P}((\Gamma)))^{\vee}
\end{equation}

If $\tau \subset \sigma$ is a face inclusion, then the restriction map $\mathcal{F}^{\mathcal{P}}_{\tau} \rightarrow \mathcal{F}^{\mathcal{P}}_{\sigma}$ is defined to be the one dual to the structure map $\mu^\mathcal{P}_{\Gamma_{\sigma} \rightarrow \Gamma_{\tau}}$, where $\Gamma_{\tau}$ and $\Gamma_{\sigma}$ are graphs corresponding to orbi-simplexes $\tau$ and $\sigma$, respectively.

For each orbi-simplex $\sigma$ we define another complex of sheaves $\mathcal{E}$ by $\mathcal{E}_{\sigma} = \Rmnum{1}[-1]$, i.e., the scalar 1 concentrated in degree 1, whose effect is to shift the gradings to the right by 1. The reason we rewrite this degree shift as a sheaf is that it will be more convenient for the comparison between our main result \eqref{main theorem} and the correspondence result for cyclic operads in the theorem \eqref{lv main theorem}  below.

Now we state the main theorem, a correspondence between Feynman transform and Verdier duality:

\begin{theorem} \label{main theorem} Let $\mathcal{P}$ be a modular $\rho$-operad, then
\begin{enumerate}
\item In the derived category of sheaves over $k$ on $\bar{Y_g}$ (and equivalently, on $M$),
\begin{equation}
\label{thm}
\mathcal{F}^{F_{\rho}\mathcal{P}} \otimes \mathcal{E} \cong D\mathcal{F}^{\mathcal{P}}
\end{equation}
\item
For any $\mathcal{G}$ in the derived category of constructible sheaves on $\bar{Y_g}$, there is a unique, up to a cocycle factor, twisted modular operad $\mathcal{P}$ such that $\mathcal{F}^{\mathcal{P}} \cong \mathcal{G}$.
\end{enumerate}
\end{theorem}

\begin{proof}
We will evaluate \eqref{thm} on an orbi-simplex $\sigma$ and establish an isomorphism, natural with respect to isomorphisms of the corresponding graph $\Gamma_{\sigma}$.

By definition,
\begin{eqnarray}
\mathcal{F}^{F_{\rho}\mathcal{P}}_{\sigma} & = & (\rho^{\vee}(\Gamma_{\sigma}) \bigotimes F_{\rho}\mathcal{P}((\Gamma_{\sigma})))^{\vee} \\
 \label{equ}       & = & (\rho^{\vee}(\Gamma_{\sigma}) \bigotimes \otimes_{v\in V(\Gamma_{\sigma})} F_{\rho}\mathcal{P}(g(v),Leg(v)))^{\vee}
\end{eqnarray}

by \eqref{MP}, \eqref{equ} becomes,
\begin{equation}
\label{equ2}
 (\rho^{\vee}(\Gamma_{\sigma}) \bigotimes \otimes_{v\in V(\Gamma_{\sigma})} \oplus_{G \in \Gamma(g(v),n(v),Leg(v))} \rho^{\vee}(G) \otimes \mathcal{P}(g(v),Leg(v))^{*})^{\vee}
\end{equation}

Let us point out a subtlety for the definitions of dualities of operads. In cyclic operads context we use a suspension of an operad $s\mathcal{O}[-1]$ instead of pure $\mathcal{O}$ in the definition of dg duality of $\mathcal{O}$, resulting in an appearance of an additional orientation factor to be addressed carefully (see theorem 3.9 in \cite{lv}). However, in modular operads context the orientation factor is contained in \eqref{orientation}, which ensures that the calculation of orientation spaces is easier (the relation between these two cases will be discussed below).

Following \cite{lv}, we note that a graph  $\Gamma_{\sigma}$ with each vertex $v$ decorated by a stable graph $G$ with $g(G)=g(v), n(G)=n(v)$ and with $Leg(G)$ the same as $Leg(v)$ is literally the same as a graph $\Gamma_{\tau}$ with a collection of subgraphs, of which the union of the set of half edges equals the set of half edges of $\Gamma_{\sigma}$, such that contracting each of these subgraphs returns $\Gamma_{\sigma}$. We call $\Gamma_{\tau}$ the \textit{vertex expansion} of $\Gamma_{\sigma}$.

Let $\Gamma_{\tau}$ be a vertex expansion of $\Gamma_{\sigma}$, and $f: \Gamma_{\tau} \rightarrow \Gamma_{\sigma}$ be the morphism corresponding to contracting all the subgraphs.

then \eqref{equ2} becomes
\begin{eqnarray}
\label{eqn2}
&& (\rho^{\vee}(\Gamma_{\sigma}) \bigotimes \oplus_{\textit{vertex expansion } \Gamma_{\tau} \textit{of } \Gamma_{\sigma}} \otimes_{v\in V(\Gamma_{\sigma})}\rho^{\vee}(f^{-1}(v)) \otimes \mathcal{P}((\Gamma_{\tau}))^{*})^{\vee}
 \end{eqnarray}

by \eqref{cocycle} and  note that $Det(E(G))[2\chi]$ is a cocycle, where $\chi=|E(G)|$, \eqref{eqn2} becomes
\begin{eqnarray*}
            =  &  \bigoplus_{\textit{vertex expansion } \Gamma_{\tau} \textit{of } \Gamma_{\sigma}} (\rho^{\vee}(\Gamma_{\tau}) \otimes \mathcal{P}((\Gamma_{\tau}))^{*})^{\vee}     \\
             = &  \bigoplus_{\textit{vertex expansion } \Gamma_{\tau} \textit{of } \Gamma_{\sigma}} (Det(E(\Gamma_{\tau}))[2m] \otimes \rho(\Gamma_{\tau})^{*} \otimes \mathcal{P}((\Gamma_{\tau}))^{*})^{\vee}\\
             = &   \bigoplus_{\textit{vertex expansion } \Gamma_{\tau} \textit{of } \Gamma_{\sigma}} Det(E(\Gamma_{\tau}))[2m]^{\vee} \otimes (\rho(\Gamma_{\tau})\otimes \mathcal{P}((\Gamma_{\tau})))^{\vee *}\\
             = &   \label{left1} \bigoplus_{\textit{vertex expansion } \Gamma_{\tau} \textit{of } \Gamma_{\sigma}} Det(E(\Gamma_{\tau}))[2m]^{\vee} \otimes  (\mathcal{F}^{\mathcal{P}}_{\tau})^{*}
\end{eqnarray*}

 where $\tau$ is an orbi-simplex corresponding to $\Gamma_{\tau}$ and $m = |E(\Gamma_{\tau})|$.

On the other hand, by theorem \eqref{orbi-verdier_duality}
\begin{eqnarray}
D\mathcal{F}^{\mathcal{P}}_{\sigma} & = & \bigoplus_{\tau \supset \sigma} (\mathcal{F}^{\mathcal{P}}_{\tau}\otimes Det(\tau)[1])^{*} \\
                & = &  \bigoplus_{\tau \supset \sigma} Det(\tau)[1]^{*} \otimes (\mathcal{F}^{\mathcal{P}}_{\tau})^{*} \\
			& = &  \label{right2} \bigoplus_{\textit{vertex expansion } \Gamma_{\tau} \textit{of } \Gamma_{\sigma}} Det^{-1}(E(\Gamma_{\tau}))[-1] \otimes (\mathcal{F}^{\mathcal{P}}_{\tau})^{*}
\end{eqnarray}
the last equality follows from $m$ = $|V(\tau)|$ = $|E(\Gamma_{\tau})|$.

Let us calculate  $(Det(E(\Gamma_{\tau}))[2m]^{\vee})^{-1}Det^{-1}(E(\Gamma_{\tau}))[-1]$:
\begin{eqnarray*}
(Det(E(\Gamma_{\tau}))[2m]^{\vee})^{-1}Det^{-1}(E(\Gamma_{\tau}))[-1] & = & Det(E(\Gamma_{\tau}))^{\vee *}[-2m]Det^{-1}(E(\Gamma_{\tau}))[-1] \\
& = & Det(E(\Gamma_{\tau}))[2m][-2m]Det^{-1}(E(\Gamma_{\tau}))[-1] \\
& = & \Rmnum{1}[-1]
\end{eqnarray*}

So we conclude that \eqref{thm} holds when evaluating on the orbi-simplex $\sigma$.
Finally, by definitions \eqref{orbi_vd_restriction} and \eqref{diffFeynman}, it is obvious that the sheaf restriction maps on both side of \eqref{thm} are the same.

Conversely, for any $\mathcal{G}$ in the derived category of constructible sheaves with values in dg-vector spaces over $k$ on $\bar{Y_g}$, there is a twisted modular $\rho$-operad $\mathcal{P}$ such that $\mathcal{F}^\mathcal{P} \cong \mathcal{G}$. In fact, it is easy to observe that the data $\mathcal{G}|_\sigma$ for every orbi-simplex $\sigma$ and the restriction morphisms $\mathcal{G}|_{\sigma} \rightarrow \mathcal{G}|_{\tau}$ for every two orbi-simplices $\sigma \subset \tau$ exactly define a modular operad that we seek (compare with the definition \eqref{modular_operad}). However, it may not be true that a twisted modular operad corresponding to $\mathcal{G}$ is unique, even up to quasi-isomorphism. The reason is that one can change the cocylce by multiplying an arbitrary another cocycle $\tau$ with the original cocycle, and then multiplying the inverse of $\tau$ with the twisted modular operad. In fact, the twisted modular operad corresponding to $\mathcal{G}$ is unique up to a cocycle factor. The reasoning is as follows: if $\mathcal{P},\mathcal{Q}$ are two twisted modular operads such that $\rho(\Gamma_\sigma)\otimes \mathcal{P}((\Gamma_\sigma)) \cong \tau(\Gamma_\sigma)\otimes \mathcal{Q}((\Gamma_\sigma)) \cong \mathcal{G}$, then  $\mathcal{Q}((\Gamma_\sigma)) \cong (\tau^{-1}(\Gamma_\sigma)\otimes\rho(\Gamma_\sigma))\otimes \mathcal{P}((\Gamma_\sigma))$ for any orbi-simplex $\sigma$ of $\bar{Y_g}$, that is, they are differed by the cocycle $\tau^{-1}\otimes \rho$.
 \end{proof}
\begin{remark}
\label{func_thm}
By a straightforward check one can see that $\mathcal{F}$ is indeed a functor from the category of Modular $\rho$-operads with values in dg-vector spaces over $k$ of characteristic 0,
denoted by $dg$-$Mod_{\rho}Op$, to the derived category of constructible sheaves on each $\bar{Y_g}$ or on $M$ (after putting all $\bar{Y_g}$ together), denoted by $DSh(\bar{Y_g})$, and the following diagram commutes:
\begin{equation*}
 \xymatrix{dg-Mod_{\rho}Op  \ar[rrr]^{F_{\rho}} \ar[d]^{\mathcal{F}}
 & & & dg-Mod_{\rho^{\vee}}Op \ar[d]^{\mathcal{F}} \\
DSh(\bar{Y_g}) \ar[rrr]^{D[1]} & & & DSh(\bar{Y_g})
}
\end{equation*}
\end{remark}
For the convenience of the following discussions, we will define two more complexes of sheaves $\mathcal{K}$ and $\mathcal{L}$ as follows:
\begin{eqnarray*}
 & \mathcal{K}_{\rho}:= Det(E(\Gamma_{\rho}))[2m], \text{where } m=|E(\Gamma_{\rho})|
\\
 & \mathcal{L}_{\rho}:= Det^{-1}(E(\Gamma_{\tau}))[-1]
\end{eqnarray*}
 for an orbi-simplex $\rho \in \bar{Y_g}$. The following simple relation holds:
\begin{equation}
 \label{main_conn}
 \mathcal{E} = \mathcal{K}^{-1} \otimes \mathcal{L}
 \end{equation}

\subsubsection{Correspondence in cyclic operads context and the deduction of it from our results}

Lazarev and Voronov(\cite{lv}) formulated a correspondence between cyclic operads and Verdier duality. We will show that one can recover their results from ours. For reader's convenience we collect their results as follows:

Assume $\mathcal{O}$ is a cyclic operad. 
We define the following rules associating certain sheaves on $\bar{Y_g}$ and $Y_g$ with $\mathcal{O}$. \label{sheaf}
\begin{itemize}
\item For an orbi-simplex $\tau \in Y_g$, we define an auxiliary orientation sheaf $\mathcal{\bar{H}}_{\tau}:= Det^{-1}(H_1(\Gamma_{\tau},k))[-n]$. For each orbi-simplex $\tau \in \bar{Y_g} \backslash Y_g$, we set $\mathcal{\bar{H}}_{\tau} = 0$. The resulting sheaf $i^{-1}\mathcal{\bar{H}}$ on ${Y_g}$ will be denoted by $\mathcal{H}$.
\item For an orbi-simplex $\tau \in Y_g$, let $\mathcal{\bar{F}}^\mathcal{O}_{\tau}$ be the sheaf associated with $\mathcal{O}$ as defined in \eqref{sheaf_asso_operad} with $\rho$ a trivial cocycle, and for any orbi-simplex $\tau \in \bar{Y_g} \backslash Y_g$, we set $\mathcal{\bar{F}}^\mathcal{O}_{\tau} = 0$. The complex of sheaves $i^{-1}\mathcal{\bar{F}}$ on $Y_g$ will be denoted by $\mathcal{F}^\mathcal{O}$.
\end{itemize}
The main result in \cite{lv} is
\begin{theorem} \label{lv main theorem}
There is a canonical isomorphism in the derived category of sheaves on ${Y_g}$
\begin{equation}
\label{lv_thm}
D\mathcal{F}^\mathcal{O} \cong \mathcal{F}^{D\mathcal{O}}\otimes \mathcal{H}[4-3n]
\end{equation}
where $D\mathcal{O}$ is the dg dual operad of $\mathcal{O}$.
\end{theorem}
We now prove that theorem \ref{lv main theorem} follows from ours. Firstly, we give a first (to our knowledge) proof of a relation between the dg dual of cyclic operads and Feynman transform of the same operads regarded as twisted modular operads. This relation has been formulated in \cite{gk3} (see formula (5.9)), however, no proof is given there and, to the best knowledge of the author, in other literature. 
A cyclic operad $\mathcal{A}$ with $\mathcal{A}(n) = 0$ for $n \leq 1$ can be regarded as a modular operad $\mathcal{\tilde{A}}$ with $\mathcal{\tilde{A}}(0,n) = \mathcal{A}(n-1), \mathcal{\tilde{A}}(g,n) = 0$ for $g\ge 1$. The requirement $\mathcal{A}(n) = 0$ for $n \leq 1$ is an inconsistency, in the sense that we don't put this constraint for general cyclic operads. We will address this issue in the last part of this paragraph. {But to say a little bit more, I believe, there exists a version of \textit{pre-modular operads}, in which the stability condition on graphs is dropped, which is much similar to the notion of pre-stable curves, and Feynman transform on pre-modular operads can be defined similarly. \textit{Any} cyclic operad can be regarded as a pre-modular operad. However, the moduli space of  (unstable) metric graphs is an infinite (noncompact) orbi-simplicial complex, the original Verdier duality is not true in this space, so we probably can not expect such a correspondence (in its form presented in this paper) in this case. We will leave it as a subject for a future work}. Now, let $F(\mathcal{\tilde{A}})$ be the Feynman transform of $\tilde{A}$ regarded as a twisted modular operad with the trivial cocycle $\rho(G) = k$. Let $Cyc(\mathcal{P})$ be the genus 0 part of the modular operad $\mathcal{P}$, i.e., $Cyc(\mathcal{P})(n) = P(0,n+1), n\ge 1$. It is a cyclic operad by definition. Moreover, let $\widetilde{Cyc}(\mathcal{P})$ be a modular operad generated by ${Cyc}(\mathcal{P})$ : $\widetilde{Cyc}(\mathcal{P})(0,n):= Cyc(\mathcal{P})(n-1)$, $\widetilde{Cyc}(\mathcal{P})(g,n):=0$, $g\geq 1, n\ge 1$. We will prove the following proposition. Here is a subtlety we must clarify: it seems that we need a constraint: $\mathcal{A}(n)=0$ for $n\leq 1$, however, as both sides of the following proposition does not use the values of $\mathcal{A}(n), n\leq 1$ (since they only evaluate on stable graphs and stable trees), the position actually holds for all cyclic operads. More precisely, we will extend the definition of $\tilde{\mathcal{A}}$ to all cyclic operads by:
\begin{eqnarray*}
	&\tilde{\mathcal{A}}(0,1)=0 \\
	&\tilde{\mathcal{A}}(0,n)=\mathcal{A}(n-1),  n\geq 2 \\
	&\tilde{\mathcal{A}}(g,n)=0, g\geq 1
\end{eqnarray*}
with the structure morphisms \eqref{m_morphisms}.
	

Again, the following proposition is the formula (5.9) from Getzler and Kapranov (\cite{gk3})
\begin{proposition}
\begin{equation}
\label{f_d}
Cyc(F(\mathcal{\tilde{A}})) \cong \Sigma sD\mathcal{A}
\end{equation}
\end{proposition}
Recall that $(sP)(g,n)  =  \Sigma^{-2(g-1)}\otimes Det^{-1}(k^n)\otimes P(g,n)$ for a modular operad $P$ and is called \textit{suspension} of $P$. When $O$ is a cyclic operad, $s\tilde{O}$ corresponds to the usual notion of suspension defined for cyclic operads $(sO)(n) = \Sigma^{1-n}O(n)$, under the usual convention correspondence: 
\begin{equation}
 O(n) = \tilde{O}(0,n+1) 
\end{equation}
i.e.,
\begin{equation}
Cyc(s\tilde{O})(n) = sO(n), \quad\text{for}\quad n \geq 2
\end{equation}
\begin{proof}
By definition (cf. \cite{gk3}), the left hand side of \eqref{f_d} is
\begin{eqnarray*}
 F(\tilde{A}) =
 & \cdots \rightarrow \Sigma_{G\in \Gamma(g,n),|G|=t}(\rho(G)\otimes \tilde{A}((G)))^*\bigotimes Det(G) \\
& \rightarrow   \Sigma_{G\in \Gamma(g,n),|G|=t+1}(\rho(G)\otimes \tilde{A}((G)))^*\bigotimes Det(G) \rightarrow \cdots
\end{eqnarray*}
where the term for $|G| = t$ is of degree $-t$.
Thereby, $Cyc(F(\tilde{A}))$ is
\begin{equation}\label{left}
\begin{split}
Cyc(F(\tilde{A})) = &
\cdots \rightarrow \Sigma_{T\in \Gamma(0,n),|T|=t}(\rho(T)\otimes \tilde{A}((T)))^*\bigotimes Det(T) \\
 & \rightarrow \Sigma_{T\in \Gamma(0,n),|T|=t+1}(\rho(T)\otimes \tilde{A}((T)))^*\bigotimes Det(T) \rightarrow \cdots
\end{split}
\end{equation}
where $\Gamma(0,n)$ denotes the set of trees $T$ such that $n(T)=n$.
The term for $|T|=t$, in which the tree $T$ has the number of edges $t$, is again, in degree $-t$.
On the other side, the dg-dual of a cyclic operad $O$ is defined in \eqref{dg_duality} as follows (cf. \cite{lv},\cite{gk})
\begin{equation}
\label{dg_dual}
\mathcal{D}(O)(n) = \Sigma_{T\in \Gamma(0,n+1)}(sO[-1]((T)))^* 
\end{equation}
which means that the dg-dual operad a cyclic operad is generated by the linear-dual of a suspension and shifted version of $O$ evaluated over trees. Using the obvious identity $A((T)) \cong \rho(T)\otimes \tilde{A}((T))$ for any tree $T$, 
we calculate $DA(n)$ as follows
\begin{equation}\label{right}
\begin{split}
DA(n) & =  \cdots \rightarrow \Sigma_{T\in \Gamma(0,n+1),|T|=t} (\rho(T)\otimes s\tilde{A}[-1]((T)))^*\bigotimes Det(T) \\
& \rightarrow \Sigma_{T\in \Gamma(0,n+1),|T|=t+1} (\rho(T)\otimes s\tilde{A}[-1]((T)))^*\bigotimes Det(T) \rightarrow \cdots
\end{split}
\end{equation}
And the degree for the term at $|T|=t$ is $-t$ as well. Note that $A$ is taking values in dg-vector spaces over $k$, thus each expression $F(\tilde{A})$ and $\mathcal{D}(A)$ is a double complex. The final gradings of $F(\tilde{A})$ and $\mathcal{D}(A)$ are the gradings of the total complex of the double complex respectively. Thus \eqref{left} and \eqref{right} only differs by grading: the latter is the shift of the former by degree $n-2$. 
{Note that the formulas \eqref{left} and \eqref{right} use different conventions for the $n$-th component: the former is modular operads convention and the latter is cyclic operads convention}. To make the convention consistent, we adopt the modular operad convention for the $n$-th component of both operads. Thereby, the corresponding degree shift on the right side of \eqref{dg_dual} for the term at $|T|=t$ is $n-3$. {From here}, 
we can use the identity $\Sigma sP(0,n)$ $\cong$ $\Sigma Det^{-1}(k^{n})[-2]\otimes P(0,n)$ for $P(0,n) = DA(n-1)$, which can cancel out the degree shift on the right hand side of \eqref{right}. This completes the proof.
\end{proof}


We can now prove that \eqref{lv main theorem} is a special case of the main theorem \eqref{main theorem}.

\begin{proposition}
The result \eqref{lv main theorem} can be derived from our main theorem \eqref{main theorem}. 
\end{proposition}


Since we are working with twisted modular operads and cyclic operads these two contexts, to emphasize their differences, we will use $\mathcal{F}_{O}$ and $\mathcal{F}_{P}$ to denote the complex of sheaves defined before associated with cyclic operads and twisted modular operads, respectively.
\begin{proof}
For a cyclic operad $A$ we regard it as a modular operad $\tilde{A}$ described above, which is also a twisted modular operad with trivial twist.
By our main theorem \eqref{main theorem}
\begin{equation}
\label{thm_2}
\mathcal{F}_P^{F(\mathcal{\tilde{A}})} \otimes \mathcal{E} \cong D\mathcal{F}_P^{\mathcal{\tilde{A}}}
\end{equation}


Let us consider $\mathcal{F}_P^{ \widetilde{Cyc}(F(\tilde{A}))}$.
By the definition of $\widetilde{Cyc}$ operator, it is straightforward to see that on $Y_g$ it is isomorphic to the first factor in the left hand side of \eqref{thm_2} (as the genus of all vertices of $Y_g$ are 0), and is 0 on $\bar{Y_g} \backslash Y_g$. For any graph $G \in \Gamma(g, n)$, if there exists a vertex $v\in V(G)$ for which the genus is greater than 0, then by the definition of Feynman transform,  $F(\tilde{A})((G))$ is 0 as $A$ is a cyclic operad. It follows that $\mathcal{F}_P^{F(\tilde{A})}$ is 0 on $\bar{Y_g} \backslash Y_g$. We thus get:
\begin{equation}
\label{connect}
i_!i^{-1}\mathcal{F}_P^{F(\tilde{A})} \cong \mathcal{F}_P^{\widetilde{Cyc}(F(\tilde{A}))} 
\end{equation}
as $\mathcal{F}_P^{\widetilde{Cyc}(F(\tilde{A}))} \cong \mathcal{F}_P^{F(\tilde{A})}$, and $i_!i^{-1}\mathcal{F}_P^{F(\tilde{A})} \cong \mathcal{F}_P^{F(\tilde{A})}$, the latter being similar to remark 3.6 in \cite{lv}.

As $F(\tilde{A})$ is modular-$R$ operad, where $R$ is the dualizing cocycle, for each orbi-simplex $\rho \in Y_g,{\mathcal{F}_P}_{\rho}^{\widetilde{Cyc}(F(\tilde{A}))} = (Det(E(\Gamma_{\rho}))[2m] \otimes {\widetilde{Cyc}(F(\tilde{A}))}((\Gamma)))^{\vee}$, where $m=|E(\Gamma_{\rho})|$, and 
 ${\mathcal{F}_O}_{\rho}^{\Sigma sD\mathcal{A}} = (\Sigma sD\mathcal{A}((\Gamma_{\rho})))^{\vee}$. We define a complex of sheaves $\mathcal{K}$ by $\mathcal{K}:= Det(E(\Gamma_{\rho}))[2m]$. Then
by \eqref{f_d},
\begin{equation*}
\mathcal{F}_O^{\Sigma sD\mathcal{A}} \otimes \mathcal{K} \cong \mathcal{F}_P^{\widetilde{Cyc}(F(\tilde{A}))}
\end{equation*}

On the right-hand side of \eqref{thm}, it is also that
\begin{equation}
\label{conn5}
 i_!i^{-1}D\mathcal{F}_P^{\tilde{A}} \cong D\mathcal{F}_O^{A}
\end{equation}
on $\bar{Y_g}$. In fact, similarly to the derivation of \eqref{connect},  we have $D\mathcal{F}_P^{\tilde{A}} \cong D\mathcal{F}_O^{A}$ and $D\mathcal{F}_O^{A} = 0$ on $\bar{Y_g} \backslash {Y_g}$,
as $\mathcal{\tilde{A}}(0,1) = 0$, $\mathcal{\tilde{A}}(0,n) = \mathcal{A}(n-1)$ for $n \ge 2$ and $\mathcal{\tilde{A}}(g,n) = 0$ for $g \geq 1, n\geq 0$.


Now from \eqref{thm} and all the isomorphisms of different complexes of sheaves established above, we have:
\begin{eqnarray}
\label{conns}
\mathcal{F}_O^{\Sigma sD\mathcal{A}} \otimes \mathcal{K} \otimes \mathcal{E} & \cong & \mathcal{F}_P^{\widetilde{Cyc}(F(\tilde{A}))} \otimes \mathcal{E} \\
& \cong & i_!i^{-1}\mathcal{F}_P^{F(\tilde{A})}\otimes \mathcal{E} \\
& \cong & i_!i^{-1}D\mathcal{F}_P^{\tilde{A}} \\
& \cong & D\mathcal{F}_O^{A}
\end{eqnarray}

Let us identify $\mathcal{F}_O^{\Sigma sD\mathcal{A}}$. Note that it is isomorphic to
\begin{equation}
\label{conn2}
\mathcal{F}_O^{\Sigma s} \otimes \mathcal{F}_O^{D\mathcal{A}}
\end{equation}
where, by a slight abuse of notations, we identify $\mathcal{F}_O^{\Sigma s}$ with $\mathcal{F}_O^{\Sigma sI}$ where $I$ is the trivial cyclic operad.
Let us calculate the first factor of \eqref{conn2}, the inverse of which is exactly computed in \cite{lv}. 
For any stable graph $\Gamma$ corresponding to an orbi-simplex $\rho$, by definition
\begin{eqnarray}
{\mathcal{F}_O}_{\rho}^{\Sigma s} & = & (\Sigma sI((\Gamma)))^{\vee} \\
\label{conn4}
& = & \bigotimes_{v\in V(\Gamma)} Det^{-1}(H(v))[-3]
\end{eqnarray}

To compute \eqref{conn4}, we follow the argument in \cite{lv} which calculates the inverse of it
\begin{eqnarray*}
\bigotimes_{v\in V(\Gamma)} Det(H(v))[3] & \cong & Det^{-3}(V(\Gamma)) \otimes \bigotimes_{v\in V(\Gamma)}Det(H(v)) \\
& \cong & Det^{-1}(V(\Gamma))[2v(\Gamma)] \otimes \bigotimes_{v\in V(\Gamma)}Det(H(v))
\end{eqnarray*}
where $v(\Gamma) = |V(\Gamma)|$. Note that the set $\prod_{v\in V(\Gamma)}H(v)$ is naturally isomorphic to
the set $\prod_{e\in E(\Gamma)}H(e)$, where $H(e)$ is the set of (two) half edges making up an edge $e$,
as both sets count the set of half-edges $H(\Gamma)$ of the graph, the former by grouping the set of half-edges by vertices, the latter by edges. By passing to determinants, we obtain

\begin{equation}
\label{determinant}
\bigotimes_{v\in V(\Gamma)} Det(H(v))[3] \cong Det^{-1}(V(\Gamma))[2v(\Gamma)] \otimes
\bigotimes_{e\in E(\Gamma)}Det(H(e))
\end{equation}
Note that the exact sequence
\begin{equation*}
0 \rightarrow H_1(\Gamma) \rightarrow C_1(\Gamma) \rightarrow C_0(\Gamma)
\rightarrow H_0(\Gamma) \rightarrow 0
\end{equation*}
yields a canonical isomorphism
\begin{equation*}
Det H_0(\Gamma) \otimes Det^{-1}H_1(\Gamma) \cong C_0(\Gamma) \otimes Det^{-1}C_1(\Gamma)
\end{equation*}
Note that the homology groups and chain complexes in the above exact sequence are manipulated similar to how cellular homology is defined, in the usual sense of regarding $\Gamma$ as a 1-dimensional CW complex.

Moreover, we have the following natural isomorphisms:
\begin{eqnarray*}
Det C_0(\Gamma) & \cong & Det V(\Gamma) \\
Det C_1(\Gamma) & \cong & \bigotimes_{e\in E(\Gamma)} Det(H(e))[1] \\
& \cong & Det^{-1}E(\Gamma)\otimes \bigotimes_{e\in E(\Gamma)}Det H(e) \\
Det(H_0(\Gamma)) & \cong & k[-1]
\end{eqnarray*}
We conclude that the last expression in \eqref{determinant} is isomorphic to
\begin{eqnarray*}
& & Det(E(\Gamma))\otimes Det(H_1(\Gamma)) \otimes Det^{-1}(H_0(\Gamma))[2v(\Gamma)] \\
&    \cong &Det^{-1}(E(\Gamma))\otimes Det(H_1(\Gamma)) \otimes Det^{-1}(H_0(\Gamma))[
2v(\Gamma)-e(\Gamma))]
\end{eqnarray*}

This conclusion immediately implies that we have an isomorphism of complexes of sheaves
\begin{equation}
\label{key}
  \mathcal{H} \cong \mathcal{F}_O^{\Sigma s} \otimes \mathcal{L}
\end{equation}
 According to the relations (of sheaves) \eqref{main_conn} and \eqref{conn2}, one can put \eqref{key} in the left hand side of \eqref{conns}, directly resulting in
\begin{equation*}
\mathcal{F}_O^{DA} \otimes \mathcal{H} \cong D\mathcal{F}_O^{A}
\end{equation*}
and completing the proof.

\end{proof}


\subsection{Applications}

One of the most important properties of Feynman transform is that it is a homotopy functor and it has a homotopy inverse. The original proof (\cite{gk3}) of this property is relatively complicated (cf. theorem(5.3) and (5.4) in \cite{gk3}). 
We demonstrate an application of our main theorem by {presenting a proof on} this homotopy property of Feynman transform in a quite simple way.

{Recall (from the subsection of notation) that the operads discussed in this paper are valued in the symmetric monoidal category of dg-vector spaces over a field $k$ of characteristic 0}.
\begin{theorem}\label{homotopy_ft}
If $\mathcal{P}$ is a modular $\rho$-operad, and $F_{\rho}$ is Feynman transform functor, then
\begin{enumerate}
\item   $F_{\rho}$ is a homotopy functor: if $f : A  \rightarrow B$ is a weak equivalence of modular $\rho$-operads (that is, it induces an isomorphism in homology), then so is $F_{\rho}f : F_{\rho}B \rightarrow  F_{\rho}A$. \label{1}
\item   Assume that, in each degree of $\mathcal{P}(g,n)$, it is a finite dimensional vector space over $k$, then the canonical map $\tau: F_{\rho^{\vee}}F_{\rho}\mathcal{P} \rightarrow \mathcal{P}$ is a weak equivalence. \label{2}
\end{enumerate}
\end{theorem}
\begin{proof}
Using main theorem \eqref{main theorem}, the proof of \eqref{1} is straightforward. Recall that $*_{g,n}$ is the graph with one vertex and no edges of genus $g$ and valence $n$. By our construction and remark \eqref{func_thm} on the naturality of the correspondence, the weak equivalence $f$ induces a weak equivalence between $\mathcal{F}^A_{\rho}|_{*_{g,n}}$ and $\mathcal{F}_{\rho}^B|_{*_{g,n}}$(in fact, by Kunneth formula, evaluating on all stable graphs in $\bar{Y_g}$ results in an isomorphism between $\mathcal{F}^A_{\rho}$ and $\mathcal{F}_{\rho}^B$ in the derived category of sheaves), then the Verdier duality also induces a weak equivalence between $D\mathcal{F}^A_{\rho}|_{*_{g,n}}$ and $D\mathcal{F}^B_{\rho}|_{*_{g,n}}$(isomorphisms in a derived category are weak equivalences), which, is exactly $F_{\rho}f$ restricted on the $(g,n)$ component by our construction, so $F_{\rho}f$ is a weak equivalence.

Through the proof of \eqref{2}, we need the following simple lemma:
\begin{lemma}
Let $V$ and $W$ be two (co)chain complexes valued in vector spaces over the field $k$
which are finite dimensional in each degree, and $f: V \rightarrow W$ be a weak equivalence between $V$ and $W$, then the induced dual map $f^\vee: W^\vee \rightarrow V^\vee$, defined by
$f^\vee (w^*)(v^*) = w^*(f(v))$ for any $w\in W, v\in V $ and their duals $ w^* \in W^\vee, v^* \in V^\vee$, is a weak equivalence between $W^\vee$ and $V^\vee$.
\end{lemma}

One can deduce this lemma from the following fact: without loss of generality we let $C$ be a chain complex:
 $\cdots \rightarrow v_{i+1} \rightarrow v_i \rightarrow v_{i-1} \rightarrow \cdots$, then there is a natural isomorphism between homology groups $H_{i}(C,k)^{\vee}$ and $H_{i}(C^{\vee},k)$. From the finite dimension property of $V$ and $W$ we know that $H_{i}(V,k)$ and $H_{i}(W,k)$ are finite dimensional for each $i\ge 0$, and then by the naturality of the functor and elementary linear algebras, it is immediately that the induced dual map $H_{i}(W^{\vee},k) \rightarrow H_{i}(V^{\vee},k)$ is an isomorphism, thereby $f^{\vee}: W^{\vee} \rightarrow V^{\vee}$ is a weak equivalence. Since the proof of this fact is straightforward, we will leave it to the readers.

As $\mathcal{P}$ is a modular $\rho$-operad, $E = F_{\rho}\mathcal{P}$ is a modular $\rho^{\vee}$-operad, therefore by the main theorem \eqref{main theorem}, we have
\begin{equation}
\mathcal{F}^{F_{\rho^{\vee}}F_{\rho}\mathcal{P}} = \mathcal{F}^{{F_{{\rho}^{\vee}}}E} \cong D\mathcal{F}^E[1] = \mathcal{F}^{ F_{\rho}\mathcal{P}}[-1][1] \cong DD\mathcal{F}^{\mathcal{P}}
\end{equation}
By our construction, $\mathcal{F}^Q$ is a constructible sheaf for any twisted modular operad $Q$,  
 so from Verdier duality \eqref{verdier_duality},  we have:
\begin{equation}
 DD\mathcal{F}^{\mathcal{P}} \cong \mathcal{F}^\mathcal{P}.
\end{equation}
Therefore, in the derived category of sheaves on $\bar{Y_g}$,
\begin{equation}
\mathcal{F}^{F_{\rho^{\vee}}F_{\rho}\mathcal{P}} \cong \mathcal{F}^\mathcal{P}
\end{equation}
Let $f: \mathcal{F}^{F_{\rho^{\vee}}F_{\rho}\mathcal{P}} \cong \mathcal{F}^\mathcal{P}$ be any such isomorphism in the derived category of sheaves on $\bar{Y_g}$. Since an isomorphism in any derived category is the one which induces an isomorphism on all cohomology groups, $f$ induces an isomorphism on all cohomology groups of the sheaves(note that this is the sheaf formed by cohomology groups of dg-vector spaces in which the sheaf takes values, rather than the sheaf cohomology). 
Notice that $*_{g,n}$ is the "fundamental building block" that generates $F_{\rho}\mathcal{P}$. The evaluation of $\mathcal{F}^{F_{\rho^{\vee}}F_{\rho}\mathcal{P}}, \mathcal{F}^\mathcal{P}$ on the corresponding orbi-simplex $\sigma_{g,n}$ of $*_{g,n}$ is $(F_{\rho^{\vee}}F_{\rho}\mathcal{P}(g, Leg(*_{g,n})))^{\vee}$ and $(\mathcal{P}(g, Leg(*_{g,n})))^{\vee}$, respectively. As $f$ induces an isomorphism between the cohomology of the sheaves, it is a weak equivalence  $f|_{\sigma_{g,n}}:  (F_{\rho^{\vee}}F_{\rho}\mathcal{P}(g, Leg(*_{g,n})))^{\vee} \rightarrow (\mathcal{P}(g, Leg(*_{g,n})))^{\vee} $. It is known that if $V$ is a chain complex over a field $k$ which has finite dimension in each degree, then $V^{\vee\vee} \cong V$,
and if $h: V \rightarrow W$ is a weak equivalence for such chain complexes $V$ and $W$, we have an induced weak equivalence $h^{\vee}: V^{\vee} \rightarrow W^{\vee}$(neither statement is true if $V$ or $W$ is not finite dimensional in some degree). Thus, under the hypothesis that any component $\mathcal{P}(g,n)$ is finite dimensional in each degree, $f$ also induces an weak equivalence $g:  F_{\rho^{\vee}}F_{\rho}\mathcal{P}(g, Leg(*_{g,n})) \rightarrow \mathcal{P}(g, Leg(*_{g,n})) $.
In particular, $f$ induces a weak equivalence
\begin{equation}
\label{weak_equiv}
F_{\rho^{\vee}}F_{\rho}\mathcal{P}(g, n) \rightarrow \mathcal{P}(g, n)
\end{equation}
Having done all the above, it remains to show that the weak equivalence \eqref{weak_equiv} induced by $f$ is indeed a morphism of modular $\rho$-operads. This is implied by the fact that the restriction maps of the associated sheaves are defined to be the dual of the structure maps of the corresponding twisted modular operads.
\end{proof}


There are now several recently proposed approaches to generalised operadic structures. All of them include modular operads as an example. Namely, an approach based on polynomial monads (\cite{bb}),  
Feynman categories (\cite{kw, kw2, bkw}) and operadic categories (\cite{bm, bm2, bmo}). It is interesting to investigate if our equivalence theorems can be extended  to  these general frameworks, and this will be left as a subject to a future work.  

\section*{Acknowledgement}
{We are}
grateful to the anonymous referee of the paper \cite{lv} who suggested
{the}
generalization {discussed in the paper}, demonstrated its
potential importance and outlined a {plan of
proving} it, although {we didn't follow that approach and have used our own}
method in this paper. We also thank Sasha Voronov for providing us the referee's report of his paper.

\end{document}